\documentclass[12,reqno]{amsart}
\usepackage{amsmath, amssymb, amsthm, mathtools}
\usepackage{url}
\usepackage[breaklinks]{hyperref}

\DeclarePairedDelimiter\floor{\lfloor}{\rfloor}

\renewcommand{\(}{\left\(}
\renewcommand{\)}{\right\)}
\renewcommand{\[}{\left\[}
\renewcommand{\]}{\right\]}
\renewcommand{\i}{\infty}
\numberwithin{equation}{section}
 \theoremstyle{plain}
\newtheorem{theorem}{Theorem}[section]
\newtheorem{lemma}[theorem]{Lemma}

\newtheorem{corollary}[theorem]{Corollary}
\newtheorem{remark}[theorem]{Remark}

   \makeatletter
\def\proof{\@ifnextchar[{\@oproof}{\@nproof}}
\def\@oproof[#1][#2]{\trivlist\item[\hskip\labelsep\textit{#2 Proof of\
#1.}~]\ignorespaces}
\def\@nproof{\trivlist\item[\hskip\labelsep\textit{Proof.}~]\ignorespaces}

\makeatother

\begin{document}
\title[$k$th smallest parts]{On the $k$th smallest part of a partition into distinct parts} 

\author{Rajat Gupta, Noah Lebowitz-Lockard, and Joseph Vandehey}\thanks{2010 \textit{Mathematics Subject Classification.} Primary 11P81, 11P82; Secondary 11P84, 05A19.\\
\textit{Keywords and phrases.} $q$-series, partition identities, smallest part, divisor function, average orders}
\address{Department of Mathematics, University of Texas at Tyler, USA.} 
\email{rgupta@uttyler.edu, nlebowitzlockard@uttyler.edu, jvandehey@uttyler.edu}

\begin{abstract} A classic theorem of Uchimura states that the difference between the sum of the smallest parts of the partitions of $n$ into an odd number of distinct parts and the corresponding sum for an even number of distinct parts is equal to the number of divisors of $n$. In this article, we initiate the study of the $k$th smallest part of a partition $\pi$ into distinct parts of any integer $n$, namely $s_k(\pi)$. Using $s_k(\pi)$, we generalize the above result for the $k$th smallest parts of partitions for any positive integer $k$ and show its connection with divisor functions for general $k$ and derive interesting special cases. We also study weighted partitions involving $s_k(\pi)$ with another parameter $z$, which helps us obtain several new combinatorial and analytical results. Finally, we prove sum-of-tails identities associated with the weighted partition function involving $s_k(\pi)$.
\end{abstract}

\maketitle
\section{Introduction}

In his second notebook \cite[p.~354]{not1} (see also \cite[p. ~264]{B}), Ramanujan states an elegant identity, namely, for $q, c \in \mathbb{C}, |q| < 1, 1-cq^{n}\neq 0$, 
\begin{align}\label{ram1}
\sum_{n=1}^\i \frac{(-1)^{n-1}c^nq^{n(n+1)/2}}{(1-q^n)(cq;q)_n}=\sum_{n=1}^\infty \frac{c^nq^n}{1-q^n},
\end{align}
where, from here on, $(a; q)_n$ is the $q$-Pochhamer symbol defined as
\begin{align*}
(a; q)_n := (1 - a)(1 - aq) \cdots (1 - aq^{n - 1}) \textup{ for } n \geq 1
\end{align*}
and similarly $(a; q)_\i = (1 - a)(1 - aq) \cdots$. (We also write $(a; q)_0 = 1$.)

Almost a century later, Uchimura \cite[Theorem $2$]{Uch} rediscovered \eqref{ram1} and gave a new representation. We state the $c = 1$ case here:
\begin{align}\label{uchi1}
\sum_{n=1}^\i nq^n(q^{n+1};q)_\i =\sum_{n=1}^\i \frac{(-1)^{n-1}q^{n(n+1)/2}}{(1-q^n)(q;q)_n}=\sum_{n=1}^\i \frac{q^n}{1-q^n}.
\end{align}
The left-hand side of the above identity \eqref{uchi1} has the following beautiful combinatorial explanation. To state that, let us first define $\mathcal{D}(n)$ to be a set of all partitions of $n$ into distinct parts, and $\# (\pi)$ and $s(\pi)$ be the length and smallest element of a given partition $\pi \in \mathcal{D}(n)$, respectively, and further define the weighted partition statistic $\textup{FFW}(n)$ (named after Fokkink, Fokkink, and Wang, who we discuss later) as
\begin{align}\label{FFWdn}
\textup{FFW}(n):=\sum_{\pi \in \mathcal{D}(n)}(-1)^{\#(\pi)-1}s(\pi),
\end{align}
for $n \geq 1$. Then we can see that 
\begin{align}\label{a}
\sum_{n=1}^{\i}\textup{FFW}(n)q^n=\sum_{n=1}^\i nq^n(q^{n+1};q)_\i.
\end{align}
From \eqref{uchi1}, \eqref{FFWdn} and \eqref{a}, we can easily deduce for $n \in \mathbb{N}$,
\begin{align}\label{FFW}
\textup{FFW}(n)=d(n).
\end{align}
The above result \eqref{FFW} was first observed by Bressoud and Subbarao \cite{BS}, and later independently by Fokkink, Fokkink, and Wang \cite{FFW}. This result recently motivated Andrews \cite{andrews} to define the smallest part function spt$(n)$, which counts the total number of appearances of the smallest parts in all partitions of $n$. Another interesting observation is that \eqref{FFW} can be considered as an analogue of Euler's Pentagonal Number Theorem \cite[Theorem $1.6$]{and}, which states that
\begin{align}\label{euler}
\sum_{\pi \in \mathcal{D}(n)} (-1)^{\# (\pi) - 1} & =  \left\{\begin{array}{ll}
(-1)^{k - 1}, & \textrm{if } n = k(3k - 1)/2 \textrm{ for some } k \in \mathbb{Z}, \\
0, & \textrm{otherwise.}
\end{array}\right.
\end{align}

Equations \eqref{uchi1} and \eqref{FFW} have been generalized in many directions, beginning with Andrews' seminal paper \cite{andrews}. Andrews, Garvan, and Liang \cite{agl} extended \eqref{FFW} by a parameter $z$ by defining
\begin{align}\label{agldef}
\textup{FFW}(z,n):=\sum_{\pi \in \mathcal{D}(n)}(-1)^{\#(\pi)-1}(1+z+z^2+\cdots z^{s(\pi)-1}),
\end{align}
with the help of its generating function they give the following extension of \eqref{uchi1}. (Note that $\textup{FFW}(1,n)=\textup{FFW}(n)$.) For $z\in \mathbb{C}$ and $|q|<1,$ we have 
\begin{align}\label{agl}
\sum_{n=1}^{\i}\textup{FFW}(z,n)q^n=\sum_{n=1}^{\infty}z^n\left(1-(q^{n+1};q)_\i \right)=\sum_{n=1}^\i \frac{(-1)^{n-1}q^{n(n+1)/2}}{(1-zq^n)(q;q)_n}=\frac{1}{1-z}\left(1-\frac{(q;q)_\i}{(zq;q)_\i} \right).
\end{align}
The second sum from the left-hand side of \eqref{agl} is an example of a sum-of-tails identity studied by Ramanujan. (The first terms in a series are the \emph{head} and the rest are the tail. The product $(q^{n + 1}; q)_\infty$ is the tail of $(q; q)_\infty$ because it does not have the first $n$ terms. For more on sum-of-tails identities, see \cite{Gup}.) While working on sum-of-tails identities, the first author \cite[Theorem 1.16]{Gup} also extended the function FFW$(n)$ by a parameter $c$, namely, 
\begin{align}\label{guptadef}
\textup{FFW}_c (n):=\sum_{\pi \in \mathcal{D}(n)}(-c)^{\#(\pi)}s(\pi),
\end{align}
and obtained an another extension of the second equality of \eqref{uchi1}, namely, for $c \in \mathbb{C}$ and $|q|<1$, we have
\begin{align}\label{gup1}
\sum_{n=1}^\i \textup{FFW}_c(n)q^n=\sum_{n=1}^\i \frac{(-c)^{n}q^{n(n+1)/2}}{(1-q^n)(q;q)_n}= \left(
\sum_{n=1}^\i \frac{
(-c;q)_n}{1-q^n}q^n- \sum_{n=1}^\i \frac{q^n}{1-q^n}\right).
\end{align}
As an application \cite[Corollary 1.17]{Gup} of the above identity \eqref{gup1} we can get the representation of the generating function of Andrews' spt-function. (For further information on applications of these extensions, see \cite{agl} and \cite{Gup}.)

The above discussion on the  importance of the \textup{FFW} function and its connection to other partition objects motivates us to study the properties of the $k$th smallest part of a partition $\pi$ of $n$, which we call $s_k (\pi)$, where $k\geq 1$. (If $\pi$ has fewer than $k$ parts, we simply set $s_k (\pi) = 0$.) 

We study $s_{k}(\pi)$ by defining an analogue of the FFW function for $s_k(\pi)$, namely, for $n \geq 1$,
\begin{align}\label{ffwk}
\textup{FFW}_{k}(n):=\sum_{\pi \in \mathcal{D}(n)} (-1)^{\# (\pi)} s_k(\pi). 
\end{align}
(Here we set $\textup{FFW}_{k}(0)=0.$) Note that the $k = 1$ case reduces to $- \textup{FFW}(n)$ defined in \eqref{FFW}. Our first main theorem states the generating function of FFW$_k(n)$, which in turn gives an extension of the equality between first and third sum of \eqref{uchi1}.
\begin{theorem}\label{maintheorem} For all $k \geq 1$, we have
\begin{align}\label{maintheoremeqn}
\sum_{n = 1}^\i \textup{FFW}_k (n) q^n & =(-1)^k q^{k(k-1)/2}\sum_{n = k}^\infty  \left[\begin{array}{c}
n - 1 \\
k - 1
\end{array}\right]_q nq^n (q^{n + 1};q)_\infty \nonumber \\
&=-\sum_{i = 0}^{k - 1} (-1)^i \frac{q^{i(i + 1)/2}}{(q;q)_i} \left(\sum_{n = 1}^\i d(n) q^n - \sum_{n = 1}^{k - i - 1} \frac{q^n}{1 - q^n}\right),
\end{align}
where the $q$-binomial coefficient is defined as
\begin{align*}
\left[\begin{array}{l}
n \\
k
\end{array}\right]_q := \frac{(q;q)_n}{(q;q)_k(q;q)_{n-k}}~~\textup{ for } n \geq k \textup{ and }\left[\begin{array}{l}
n \\
k
\end{array}\right]_q:=0 \textup{ for } n < k.
\end{align*}
\end{theorem}
If we set $k=1$ in \eqref{maintheoremeqn} and compare the coefficient of $q^n$ for $n \geq 1$, we obtain \eqref{FFW}. However, if we take $k=2, 3$ and compare the coefficient of $q^n$, we obtain interesting explicit representations of $\textup{FFW}_{2}(n)$ and $\textup{FFW}_3 (n)$.
\begin{corollary}
For $n \in \mathbb{N},$
\begin{align}
\textup{FFW}_{2}(n) & =\sum_{j=1}^{n-1}d(j)-d(n)+1, \\
\textup{FFW}_3 (n) & = -\sum_{i = 1}^{n - 1} \left(\floor*{\frac{n - i - 1}{2}} - 1\right) d(i) - d(n) - n + \frac{1}{2} (-1)^n + \frac{5}{2}.
\end{align}
\end{corollary}

Using Theorem \ref{maintheorem}, we also found an asymptotic formula for FFW$_k (n)$ for all $k > 1$.

\begin{theorem} \label{asymptotic} Fix $k > 1$. As $n \to \infty$, we have
\begin{align*}
\textup{FFW}_k (n) & \sim \frac{(-1)^k}{(k - 1)!} n^{k - 1} \log n.
\end{align*}
\end{theorem}

Before discussing further results, we note that Dilcher \cite{dil} studied an analogue of the second summation in \eqref{maintheoremeqn} by having the usual binomial coefficient instead of the $q$-binomial coefficient. He \cite[Lemma 1, Corollary 2]{dil} states that for $k \geq 1$ and $|q|<1$, 
\begin{align*}
\sum_{n=k}^{\infty}\binom{n}{k}q^n(q^{n+1};q)_\infty&=q^{k(k-1)/2}\sum_{m=1}^{\infty}\frac{(-1)^{m-1}q^{(m+k)(m+k-1)/2}}{(q;q)_{m}(1-q^{m})^k}\nonumber\\
&=\sum_{j_1=1}^{\infty}\frac{q^{j_1}}{1-q^{j_1}}\sum_{j_2=1}^{j_1}\frac{q^{j_2}}{1-q^{j_2}}\cdots\sum_{j_k=1}^{j_{k-1}}\frac{q^{j_k}}{1-q^{j_k}}.
\end{align*}

We will now turn our attention to extend the first equality of \eqref{uchi1}. Indeed, we prove the following extension of \eqref{uchi1} for $k \geq 1.$ 
\begin{theorem}\label{mainth}
For $k\geq 1$, and $|q|<1,$ we have
\begin{align}
\sum_{n = 1}^\i \textup{FFW}_k (n) q^n=\sum_{j=0}^{k-1}\sum_{m=k}^{\infty}\frac{(-1)^mq^{m(m+1)/2}}{(q;q)_{m}(1-q^{m-j})}.
\end{align}
\end{theorem}

\begin{remark}
A combination of Theorems \ref{maintheorem} and \ref{mainth} gives the full-fledged extension of \eqref{uchi1}, namely:
\begin{align*}
\sum_{n = 1}^\i \textup{FFW}_k (n) q^n=\sum_{j=0}^{k-1}\sum_{m=k}^{\infty}\frac{(-1)^mq^{m(m+1)/2}}{(q;q)_{m}(1-q^{m-j})}=-\sum_{i = 0}^{k - 1} (-1)^i \frac{q^{i(i + 1)/2}}{(q;q)_i} \left(\sum_{n = 1}^\i d(n) q^n - \sum_{n = 1}^{k - i - 1} \frac{q^n}{1 - q^n}\right).
\end{align*}
 
\end{remark}

In light of Andrews, Garvan, and Liang's \cite{agl} extension of \eqref{FFW}, we explored another important extension of \eqref{FFWdn}. For $z \in \mathbb{C}$, we introduce and define 
\begin{align}\label{glvkz1}
\text{FFW}_k(z,n):=\sum_{\pi \in \mathcal{D}(n)} (-1)^{\# (\pi)} z^{s_k(\pi)}.
\end{align} 
The complete study of this function is given in Section \ref{General Identities}. Indeed, our investigation of this function led us to some new and interesting results. As one of the application of theorems associated with \eqref{glvkz1}, we obtain the following companion to an identity of Alladi \cite{alladi} (see \eqref{alladi}), namely, 
\begin{theorem}\label{121}
For $n,j \in \mathbb{N}$ and $k \in \mathbb{Z}\backslash\{0\}$, we have,
\begin{align}
\sum_{\substack{\pi \in \mathcal{D}(n)\\ s(\pi) \textup{ even}}} (-1)^{\# (\pi)} =
\left\{
	\begin{array}{ll}
		(-1)^k-(-1)^{j} & \mbox{if } n=j^2=\frac{k(3k\pm 1)}{2}, \\
		(-1)^{j-1}  & \mbox{if } n = j^2 \mbox{ and } n \mbox{ is not pentagonal}, \\
		(-1)^k & \mbox{if } n = \frac{k(3k\pm 1)}{2} \mbox{ and } n \mbox { is not square}, \\
		0 & \mbox{otherwise}.
	\end{array}
\right.
\end{align}
\end{theorem}
This result was also given in \cite{OSW}, but out motivation and methods are completely different. Indeed, our methods show that these kind of results belong to a more general theory. In Section \ref{General Identities}, we obtain general theorems, such as Theorem \ref{121}, and their combinatorial interpretations. At the end of Section \ref{General Identities}, we obtain the closed form of Uchimura's formula for the sum of fixed powers of the $s(n)$ \cite{Uch2}. We also find a new representation of Agarwal et al.'s \cite{abem} more general sum
\begin{align}\label{maji0}
M_{m,c}:=\sum_{n=1}^{\infty}n^mc^nq^n(q^{n+1};q)_\infty.
\end{align} 
Finally in Section \ref{Sum-of-tails Identities}, we conclude our paper with proving few general sum-of-tails identities. Indeed, we are able to find a sum-of-tails version of Andrews et al.'s weighted partition statistic \eqref{agldef}.

\begin{theorem} \label{z tails-1} For all $k \geq 2$ and $z \in \mathbb{C}$, we have
\begin{align*}
\sum_{n = 1}^\i \left(\sum_{\substack{\pi \in \mathcal{D} (n) \\ \#(\pi) \geq k}} (-1)^{\# (\pi)} \left(1 + z + z^2 + \cdots + z^{s_k (\pi) - s_{k - 1} (\pi) - 1}\right)\right) q^n \\
= (-1)^{k - 1} \frac{q^{(k - 1)(k - 2)/2}}{z} \sum_{m = k - 1}^\i \left(\frac{q}{z}\right)^m \left[\begin{array}{c}
m - 1 \\
k - 2
\end{array}\right]_q \sum_{n = m + 1}^\i z^n ((q^n; q)_\i - 1).
\end{align*}
\end{theorem}

\section{Proofs of the main theorems and their corollaries}

To prove our main theorems we need preliminary lemmas. From here on, we let $d(n)$ be the number of divisors of $n$ and $d_{\geq k} (n)$ be the number of divisors of $n$ which are $\geq k$.

\begin{lemma}[{\cite[p. $59$]{Mac}}]\label{lemma1} For a given integer $n$, we have
\begin{align*}
\sum_{0 < n_1 < n_2 < \cdots < n_k \leq n} q^{n_1 + n_2 + \cdots + n_k} = q^{k(k+1)/2} \left[\begin{array}{c}
n \\
k
\end{array}\right]_q,
\end{align*}
where $\left[\begin{array}{c}
n \\
k
\end{array}\right]_q$ is the $q$-binomial coefficient defined in Theorem \ref{maintheorem}.
\end{lemma}
\begin{lemma}\label{lemma2} For all $k \geq 1$, we have
\begin{align*}
\sum_{n = 1}^\infty \frac{nq^{kn}}{(q; q)_n} = \frac{1}{(q^k; q)_\infty} \sum_{n = k}^\infty d_{\geq k} (n) q^n.
\end{align*}
\end{lemma}
\begin{proof} We start with the following relation \cite[Equation $(6.21)$]{Fin}:
\begin{align}\label{zn}
\sum_{n=1}^{\i}\frac{z^{n}}{(q; q)_{n}}=\frac{1}{(z; q)_\i}.
\end{align}
Differentiating with respect to $z$ gives us
\begin{align}
\sum_{n=1}^{\i}\frac{nz^{n-1}}{(q; q)_{n}}=\frac{1}{(z; q)_{\i}}\sum_{n=0}^{\i}\frac{q^{n}}{1-zq^{n}},\label{zsum}
\end{align}
Finally, we plug in $z = q^k$ and multiply by $q^k$ to obtain
\begin{align}
\sum_{n = 1}^\infty \frac{nq^{kn}}{(q; q)_n} = \frac{q^k}{(q^k; q)_\infty} \sum_{n = 0}^\infty \frac{q^n}{1 - q^{n + k}} = \frac{1}{(q^k; q)_\infty} \sum_{n = 0}^\infty \frac{q^{n + k}}{1 - q^{n + k}} = \frac{1}{(q^k; q)_\infty} \sum_{n = k}^\infty \frac{q^n}{1 - q^n}.
\end{align}
For a given integer $m$, we have $q^m/(1 - q^m) = q^m + q^{2m} + \cdots$. The coefficient of $q^n$ in this sum is equal to $1$ if $m | n$ and $0$ otherwise. Therefore,
\begin{align}
\sum_{n = k}^\infty \frac{q^n}{1 - q^n} = \sum_{n = k}^\infty \sum_{m : n | m} q^m = \sum_{m = 1}^\infty q^m \sum_{\substack{n \geq k \\ n | m}} 1 = \sum_{n = 1}^\infty d_{\geq k} (m) q^m.
\end{align}
\end{proof}
We can prove the following result.
\begin{lemma}\label{lemma3} For all $k \geq 1$, we have
\begin{align*}
\sum_{n = 1}^\infty \left[\begin{array}{c}
n - 1 \\
k - 1
\end{array}\right]_q \frac{nq^n}{(q; q)_n} & = -\frac{q^k}{(q;q)_\infty} \sum_{i = 1}^k (-1)^i \frac{q^{i(i-1)/2 - ki}}{(q; q)_{k - i}} \sum_{n = i}^\infty d_{\geq i} (n) q^n.
\end{align*}
\end{lemma}

\begin{proof} The overall structure of the proof is as follows. First, we rewrite the binomial coefficient on the lefthand side as a sum. Then, we interchange that sum with the sum over $n$. We have
\begin{align*}
\left[\begin{array}{c}
n - 1 \\
k - 1
\end{array}\right]_q & = \frac{(1 - q^{n - 1}) (1 - q^{n - 2}) \cdots (1 - q^{n - k + 1})}{(q; q)_{k - 1}} \\
& = \frac{1}{(q; q)_{k - 1}} \sum_{i = 0}^{k - 1} (-1)^i \sum_{n - k + 1 \leq n_1 < n_2 < \cdots < n_i \leq n - 1} q^{n_1 + n_2 + \cdots + n_i}.
\end{align*}
Lemma \ref{lemma1} implies that this quantity is equal to
\begin{align*}
\frac{1}{(q; q)_{k - 1}} \sum_{i = 0}^{k - 1} (-1)^i q^{(n - k)i + i(i + 1)/2} \left[\begin{array}{c}
k - 1 \\
i
\end{array}\right]_q = \sum_{i = 0}^{k - 1} (-1)^i \frac{q^{(n - k)i + i(i + 1)/2}}{(q; q)_i (q; q)_{k - i - 1}}.
\end{align*}
(Note that if $n < k$, these quantities equal $0$.)

From here, we consider the following sum:
\begin{align*}
\sum_{n = 1}^\infty \left[\begin{array}{c}
n - 1 \\
k - 1
\end{array}\right]_q \frac{nq^n}{(q; q)_n} & = \sum_{n = 1}^\infty \left(\sum_{i = 0}^{k - 1} (-1)^i \frac{q^{(n - k)i + i(i + 1)/2}}{(q; q)_i (q; q)_{k - i - 1}}\right) \frac{nq^n}{(q; q)_n} \\
& = \sum_{i = 0}^{k - 1} (-1)^i \frac{q^{i(i + 1)/2 - ki}}{(q; q)_i (q; q)_{k - i - 1}} \sum_{n = 1}^\infty \frac{nq^{(i + 1)n}}{(q; q)_n}.
\end{align*}
Lemma \ref{lemma2} implies that
\begin{align*}
\sum_{n = 1}^\infty \frac{nq^{(i + 1) n}}{(q; q)_n} & = \frac{1}{(q^{i + 1}; q)_\infty} \sum_{n = i + 1}^\infty d_{\geq i + 1} (n) q^n.
\end{align*}
Hence,
\begin{align*}
\sum_{n = 1}^\infty \left[\begin{array}{c}
n - 1 \\
k - 1
\end{array}\right]_q \frac{nq^n}{(q; q)_n}
& = \frac{1}{(q; q)_\infty} \sum_{i = 0}^{k - 1} (-1)^i \frac{q^{i(i + 1)/2 - ki}}{(q; q)_{k - i - 1}} \sum_{n = i + 1}^\infty d_{\geq i + 1} (n) q^n \nonumber \\
& = -\frac{q^k}{(q; q)_\infty} \sum_{i = 1}^k (-1)^i \frac{q^{i(i - 1)/2 - ki}}{(q; q)_{k - i}} \sum_{n = i}^\infty d_{\geq i} (n) q^n.
\end{align*}
\end{proof}

Now we are in a position to prove our main theorems.
\begin{proof}[Theorem \ref{maintheorem}][] If $n$ is the $k$th smallest part of a partition into distinct parts $\pi$, then $\pi$ contains a sequence of distinct numbers $n_1, n_2, \ldots, n_{k - 1}$ all less than $n$. All other elements are greater than $n$. Therefore,
\begin{align}
\sum_{n = 1}^\infty \textup{FFW}_k (n) q^n & = (-1)^k \sum_{n = k}^\infty \left(\sum_{0 < n_1 < n_2 < \cdots < n_{k - 1} < n} q^{n_1 + n_2 + \cdots + n_{k - 1}}\right) nq^n (q^{n + 1};q)_\infty \nonumber
\end{align}
Employing Lemma \ref{lemma1} gives us
\begin{align}\label{eqn1}
\sum_{n = 1}^\infty \textup{FFW}_k (n) q^n & = (-1)^k \sum_{n = k}^\infty q^{k(k-1)/2} \left[\begin{array}{c}
n - 1 \\
k - 1
\end{array}\right]_q nq^n (q^{n + 1}; q)_\infty \\
& = (-1)^k q^{k(k-1)/2} (q)_\infty \sum_{n = k}^\infty \left[\begin{array}{c}
n - 1 \\
k - 1
\end{array}\right]_q \frac{nq^n}{(q;q)_n}\nonumber.
\end{align}
From Lemma \ref{lemma3} we deduce that 
\begin{align*}
\sum_{n = 1}^\infty \textup{FFW}_k (n) q^n & = -(-1)^k q^{k(k+1)/2} \sum_{i = 1}^k (-1)^i \frac{q^{i(i-1)/2 - ki}}{(q;q)_{k - i}} \sum_{n = i}^\infty d_{\geq i} (n) q^n.
\end{align*}
Taking into account the simple observation $\frac{k(k+1)}{2} +\frac{i(i-1)}{2}- ki = \frac{(k-i)(k-i+1)}{2}$, we get
\begin{align*}
\sum_{n = 1}^\infty \textup{FFW}_k (n) q^n & = -(-1)^k \sum_{i = 1}^k (-1)^i \frac{q^{(k-i)(k-i+1)/2}}{(q;q)_{k - i}} \sum_{n = i}^\infty d_{\geq i} (n) q^n.
\end{align*}
Replacing $i$ with $k - i$ gives us
\begin{align*}
\sum_{n = 1}^\infty \textup{FFW}_k q^n & = -(-1)^k \sum_{i = 0}^{k - 1} (-1)^{k - i} \frac{q^{i(i + 1)/2}}{(q; q)_i} \sum_{n = k - i}^\infty d_{\geq k - i} (n) q^n.
\end{align*}
It is also possible to rewrite this result so that it only uses the divisor function $d(n)$, rather than the variant $d_{\geq k - i}(n)$. Specifically, we have
\begin{align*}
\sum_{n = i}^\infty d_{\geq k - i} (n) q^n = \sum_{n = 1}^\infty d(n) q^n - \sum_{n = 1}^{k - i - 1} \frac{q^n}{1 - q^n}, 
\end{align*}
and hence our main result, 
\begin{align*}
\sum_{n = 1}^\infty \textup{FFW}_k (n) q^n & = -\sum_{i = 1}^k (-1)^i \frac{q^{i(i+1)/2}}{(q;q)_{i}} \left( \sum_{n = 1}^\infty d(n) q^n - \sum_{n = 1}^{k - i - 1} \frac{q^n}{1 - q^n}\right).
\end{align*}
\end{proof}

Next we prove our second main theorem.
\begin{proof}[Theorem \textup{\ref{mainth}}][]
From \eqref{eqn1} we have
\allowdisplaybreaks{\begin{align}\label{proofbig}
\sum_{n = 1}^\i \textup{FFW}_k (n) q^n&=(-1)^kq^{\frac{k(k-1)}{2}}\sum_{n=1}^{\infty}n\left[\begin{array}{c}
n - 1 \\
k - 1
\end{array}\right]_qq^{n}(q^{n+1};q)_\infty.
\end{align}}
Note that \cite[Equation $(2.2.6)$]{and} states that for $|z|,|q|<1$, we have
\begin{align*}
(z;q)_\infty=\sum_{m=0}^{\infty}\frac{(-z)^mq^{m(m-1)/2}}{(q;q)_m}
\end{align*}
Substituting $z = q^{n + 1}$ and plugging this result back into Equation \ref{proofbig} gives us
\begin{align}
\sum_{n = 1}^\i \textup{FFW}_k (n) q^n&=(-1)^kq^{\frac{k(k-1)}{2}}\sum_{n=1}^{\infty}n\left[\begin{array}{c}
n - 1 \\
k - 1
\end{array}\right]_qq^{n}\sum_{m=0}^{\infty}\frac{(-1)^mq^{m(n+1)}q^{m(m-1)/2}}{(q;q)_m} \nonumber\\
&=(-1)^kq^{\frac{k(k-1)}{2}}\sum_{m=0}^{\infty}\frac{(-1)^mq^{m(m+1)/2}}{(q;q)_m}\sum_{n=k}^{\infty}n\left[\begin{array}{c}
n - 1 \\
k - 1
\end{array}\right]_qq^{n(m+1)} \nonumber\\
&=(-1)^kq^{\frac{k(k-1)}{2}}\sum_{m=0}^{\infty}\frac{(-1)^mq^{m(m+1)/2}}{(q;q)_m}\sum_{n=0}^{\infty}(n+k)\left[\begin{array}{c}
n +k- 1 \\
k - 1
\end{array}\right]_qq^{(n+k)(m+1)} \nonumber\\
&=(-1)^kq^{\frac{k(k+1)}{2}}\sum_{m=0}^{\infty}\frac{(-1)^mq^{m(m+1)/2+km}}{(q;q)_m}\sum_{n=0}^{\infty}(n+k)\left[\begin{array}{c}
n +k- 1 \\
n
\end{array}\right]_qq^{n(m+1)} \nonumber\\
&=(-1)^kq^{\frac{k(k+1)}{2}}\sum_{m=0}^{\infty}\frac{(-1)^mq^{m(m+1)/2+km}}{(q;q)_m}\Bigg(\sum_{n=0}^{\infty}n\left[\begin{array}{c}
n +k- 1 \\
n
\end{array}\right]_qq^{n(m+1)}\nonumber\\
&\qquad\qquad \qquad\qquad\qquad+k\sum_{n=0}^{\infty}\left[\begin{array}{c}
n +k- 1 \\
n
\end{array}\right]_qq^{n(m+1)}\Bigg),
\end{align}
where in the penultimate step we have employed \cite[Equation $(3.3.2)$]{and}, 
\begin{align*}
\left[\begin{array}{c}
n +k- 1 \\
k-1
\end{array}\right]_q=\left[\begin{array}{c}
n +k- 1 \\
n
\end{array}\right]_q.
\end{align*}
We now evaluate the inner two sums in the right-hand side of the previous equation. To do that, we will use \cite[Equation (3.3.7)]{and}, that is for $N \in \mathbb{N},$ 
\begin{align}\label{Nbino}
\sum_{j=0}^{\infty}\left[\begin{array}{c}
N +j- 1 \\
j
\end{array}\right]_q z^{j}=\frac{1}{(z;q)_N}.
\end{align}
Differentiating with respect to $z$, we obtain
\begin{align}\label{Nbinodiff}
\sum_{j=1}^{\infty}j\left[\begin{array}{c}
N +j- 1 \\
j
\end{array}\right]_q z^{j-1}=\frac{d}{dz}\frac{1}{(z;q)_N}=\frac{1}{(z;q)_N}\sum_{j=0}^{N-1}\frac{q^j}{1-zq^{j}}.
\end{align}
Employing \eqref{Nbino} and \eqref{Nbinodiff} with $z=q^{m+1}$ in \eqref{proofbig}, we obtain, 
\begin{align*}
\sum_{n = 1}^\i \textup{FFW}_k (n) q^n
&=(-1)^kq^{\frac{k(k+1)}{2}}\sum_{m=0}^{\infty}\frac{(-1)^mq^{m(m+1)/2+km}}{(q;q)_m}\Bigg(\frac{1}{(q^{m+1};q)_k}\sum_{j=0}^{k-1}\frac{q^{m+j+1}}{1-q^{m+j+1}}+\frac{k}{(q^{m+1};q)_k}\Bigg)\nonumber\\
&=(-1)^kq^{\frac{k(k+1)}{2}}\sum_{m=0}^{\infty}\frac{(-1)^mq^{m(m+1)/2+km}}{(q;q)_{m+k}}\Bigg(\sum_{j=0}^{k-1}\frac{q^{m+j+1}}{1-q^{m+j+1}}+k\Bigg)\nonumber\\
&=(-1)^{k+1}q^{\frac{k(k-1)}{2}}\sum_{m=1}^{\infty}\frac{(-1)^mq^{m(m-1)/2+km}}{(q;q)_{m+k-1}}\Bigg(\sum_{j=0}^{k-1}\frac{q^{m+j}}{1-q^{m+j}}+\sum_{j=0}^{k-1}1\Bigg)\nonumber\\
&=(-1)^{k+1}q^{\frac{k(k-1)}{2}}\sum_{m=1}^{\infty}\frac{(-1)^mq^{m(m-1)/2+km}}{(q;q)_{m+k-1}}\Bigg(\sum_{j=0}^{k-1}\frac{1}{1-q^{m+j}}\Bigg)\nonumber\\
&=\sum_{j=0}^{k-1}\sum_{m=1}^{\infty}\frac{(-1)^{m+k-1}q^{(m+k)(m+k-1)/2}}{(q;q)_{m+k-1}(1-q^{m+j})}.
\end{align*}
Replacing $j$ by $k-j+1$ and then $m$ by $m-k+1$, we obtain,
\begin{align*}
\sum_{n = 1}^\i \textup{FFW}_k (n) q^n&=\sum_{j=0}^{k-1}\sum_{m=k}^{\infty}\frac{(-1)^{m}q^{m(m+1)/2}}{(q;q)_{m}(1-q^{m-j})}.
\end{align*}
This concludes the proof.
\end{proof}

For $k \geq 3$, the generating functions of FFW$_{k}(n)$ get increasingly complicated, but we can still obtain asymptotic formulae for FFW$_{k}(n)$. Suppose we fix $k$ and let $n \to \infty$. The ``dominant" term in Theorem \ref{maintheorem} is
\begin{align*}
(-1)^k \frac{q^{k(k - 1)/2}}{(q;q)_{k - 1}} \sum_{n = 1}^\infty d(n) q^n.
\end{align*}
Note that $1/(q; q)_{k - 1}$ is the generating function for the number of partitions of a number into parts $< k$. For a fixed value of $k$, we have
\begin{align} \label{ffwk sum}
\textup{FFW}_k (n) \sim (-1)^k \sum_{i = 0}^{n - (k(k - 1)/2)} p_{< k} (i) d\left(n - i - \frac{k(k - 1)}{2}\right),
\end{align}
where $p_{< k} (i)$ is the number of partitions of $i$ into parts which are less than $k$. In addition \cite[Theorem 4.3]{and}, we have
\begin{align} \label{p<k sum}
p_{< k} (n) \sim \frac{n^k}{k! (k - 1)!}
\end{align}
as $n \to \infty$. Using partial summation, we obtain an asymptotic formula for the convolution of $p_{< k}$ and $d$, allowing us to approximate $\textup{FFW}_k (n)$ for large $n$.

\begin{proof}[Theorem \textup{\ref{asymptotic}}][]
Note that by \eqref{ffwk sum}, it suffices to show that
\begin{align*}
\sum_{i = 0}^n p_{< k} (i) d(n - i) \sim \frac{n^{k - 1} \log n}{(k - 1)!^2}.
\end{align*}
In particular, the shift from $n - k(k - 1)/2$ to $n$ will not affect the asymptotic.

By reindexing, we have
\begin{align*}
\sum_{i = 0}^n p_{< k} (i) d(n - i) = \sum_{i = 0}^n p_{< k} (n - k) d(i).
\end{align*}
We would like to use \eqref{p<k sum} to replace $p_{< k} (n - i)$ with $(1 + o(1)) \frac{1}{(k - 1)! (k - 2)!} (n - i)^{k - 2}$; however, this estimate is not uniform across the interval. For any $\epsilon > 0$, there exists an $i_0 \in \mathbb{N}$ such that $o(1)$ can be replaced with $O(\epsilon)$ (with implicit constant $1$) for all $i \geq i_0$. For $i < i_0$, we will replace $p_{< k} (n - i)$ with $\frac{1}{(k - 1)! (k - 2)!} (n - i)^{k - 2} + O_\epsilon (1)$, where the implicit constant is the maximum distance between $p_{< k} (n - i)$ and $\frac{1}{(k - 1)! (k - 2)!} (n - i)^{k - 2}$ for $i < i_0$.

Thus, we have
\begin{align}
\sum_{i = 0}^n p_{< k} (i) d(n - i) = & \sum_{i = 0}^n \frac{1}{(k - 1)! (k - 2)!} (n - i)^{k - 2} d(i) \label{first sum} \\
& + O\left(\epsilon \sum_{i = i_0}^n \frac{1}{(k - 1)! (k - 2)!} (n - i)^{k - 2} d(i)\right) \label{second sum} \\
& + O_\epsilon \left(\sum_{i = 0}^{i_0 - 1} d(i)\right). \label{third sum}
\end{align}
Note that we can pull the big-$O$ terms out of the summations because they are now uniform.

The first summation \eqref{first sum} will require us to use partial summation as well as the well-known estimate $\sum_{i \leq m} d(i) = m \log m + m(2\gamma - 1) + O(\sqrt{m})$. We thus have
\begin{align*}
& \sum_{i = 0}^n \frac{1}{(k - 1)! (k - 2)!} (n - i)^{k - 2} d(i) & \\
& \quad = \frac{1}{(k - 1)! (k - 2)!} \left((n - n)^k \sum_{i = 0}^n d(i) - \int_0^n \left(-(k - 2)(n - u)^{k - 3} \sum_{i \leq u} d(i)\right) du\right) \\
& \quad = \frac{1}{(k - 1)! (k - 3)!} \int_0^n (n - u)^{k - 3} \sum_{i \leq u} d(i) du \\
& \quad = \frac{1}{(k - 1)! (k - 3)!} \left(\int_0^n ((n - u)^{k - 3} (u \log u + u(2\gamma - 1))) du + O\left(\int_0^n (n - u)^{k - 3} \sqrt{u} du\right)\right) \\
& \quad = \frac{1}{(k - 1)! (k - 3)!} \left(\frac{n^{k - 1} (\log n + O(1))}{(k - 2)(k - 1)} + O(n^{k - 3/2})\right) \\
& \quad = \frac{n^{k - 1} (\log n + O(1))}{(k - 1)!^2}.
\end{align*}
The integrals in the penultimate step can all be evaluated directly using, for example, integration by parts. We have elided over the details.

The summation \eqref{second sum} is at most $O(\epsilon)$ times \eqref{first sum}. The third summation \eqref{third sum} is simply $O_\epsilon (1)$. In total, we have
\begin{align*}
\sum_{i = 0}^n p_{< k} (i) d(n - i) = \frac{n^{k - 1} \log n}{(k - 1)!^2} + O\left(\epsilon \frac{n^{k - 1} \log n}{(k - 1)!^2}\right) + O(n^{k - 1}) + O_\epsilon (1).
\end{align*}
We note that by taking $\epsilon \to 0$ sufficiently slowly as $n \to \infty$, all the asymptotic terms here, including $O_\epsilon (1)$, will be $o(n^{k - 1} \log n)$, which completes the proof.
\end{proof}

\section{General Identities}\label{General Identities}

In this section we study a new analogue of the FFW-function defined in \eqref{FFWdn}. We define this function by
\begin{align}\label{glvkz}
\text{FFW}_k(z,n):=\sum_{\pi \in \mathcal{D}(n)} (-1)^{\# (\pi)} z^{s_k(\pi)}.
\end{align} 
However, we were naturally led to this function by exploring the generalization of FFW$(z,n)$ of Andrews, Garvan, and Liang, given in \eqref{agldef}. We will briefly sketch the proof of the generating function of \eqref{glvkz} since the proof is along very similar lines to those of our above main theorems in the previous section 

Let us state the generating function of \eqref{glvkz}.
\begin{theorem}\label{Theoremgen1}
For $z \in \mathbb{C}$, $k \in \mathbb{N}$, and $|q|<1,$ we have
\begin{align}
\sum_{n=1}^{\infty}\textup{FFW}_k(z,n)q^n=\sum_{n = 1}^\i \left(\sum_{\pi \in \mathcal{D}(n)} (-1)^{\# (\pi)} z^{s_k(\pi)}\right) q^n=z^k
\sum_{n=k}^{\infty}\frac{(-1)^{n}q^{n(n+1)/2}}{(q;q)_{n-k}(zq^{n-k+1};q)_k}
\end{align}
\end{theorem}

\begin{proof}
Similar to the proof of Theorem \ref{mainth}, we can show that, 
\begin{align}
\sum_{n = 1}^\i \textup{FFW}_k(z,n)q^n&=(-1)^k q^{\frac{k(k-1)}{2}}\sum_{n = k}^\infty  \left[\begin{array}{c}
n - 1 \\
k - 1
\end{array}\right]_q (zq)^n (q^{n + 1};q)_\infty\nonumber\\
&=(-1)^kq^{\frac{k(k-1)}{2}}\sum_{m=0}^{\infty}\frac{(-1)^mq^{m(m+1)/2}}{(q;q)_m}\sum_{n=k}^{\infty}\left[\begin{array}{c}
n - 1 \\
k - 1
\end{array}\right]_q(zq^{m+1})^{n}\nonumber\\
&=(-z)^kq^{\frac{k(k+1)}{2}}\sum_{m=0}^{\infty}\frac{(-1)^mq^{m(m+1)/2+mk}}{(q;q)_m}\sum_{n=0}^{\infty}\left[\begin{array}{c}
n+k - 1 \\
n
\end{array}\right]_q(zq^{m+1})^{n}\nonumber\\
&=(-z)^kq^{\frac{k(k+1)}{2}}\sum_{m=0}^{\infty}\frac{(-1)^mq^{m(m+1)/2+mk}}{(q;q)_m(zq^{m+1};q)_k}.
\end{align}
In the last step we have used \eqref{Nbino}. The proof will be complete after replacing $m \to n-k.$
\end{proof}

We will give a different generating function for FFW$_{k}(z,n)$. 
\begin{theorem}\label{Theoremgen2}
For $z \in \mathbb{C}$, $k \in \mathbb{N}$, and $|q|<1,$ we have
\begin{align}
\sum_{n = 1}^\i \textup{FFW}_k(z,n)q^n=(-1)^kq^{\frac{k(k-1)}{2}}(q;q)_\infty\sum_{i=0}^{k-1}\frac{(-1)^i q^{\frac{i(i+1)}{2}-ik}}{(q;q)_i (q;q)_{k-i-1}}\left(\frac{1}{(zq^{i+1};q)_\infty}-\sum_{j=0}^{k-1}\frac{(zq^{i+1})^j}{(q;q)_j}\right).
\end{align}
\end{theorem}
\begin{proof}
The proof is very similar to the proof of Theorem \ref{maintheorem}. 
\end{proof}

Let us study some of the interesting special cases  of the above theorems which we just stated and derived. We obtain Euler's pentagonal number theorem \eqref{euler} by putting $z=1$ and $k=1$ in Theorems \ref{Theoremgen1} and \ref{Theoremgen2}. However, letting $z=-1$ and $k=1$ in these theorems gives us the following corollary for the smallest part of a partition into distinct parts.
\begin{corollary}\label{coro1}
If $|q|<1$, then
\begin{align}\label{coro1eqn}
\sum_{n = 1}^\i \left(\sum_{\substack{\pi \in \mathcal{D}(n)}} (-1)^{\# (\pi)-s(\pi)} \right) q^n=\sum_{n=1}^{\infty}\frac{(-1)^{n-1}q^{n(n+1)/2}}{(q;q)_{n-1}(1+q^n)}=(q;q)_{\infty}-\frac{(q;q)_\infty}{(-q;q)_{\infty}}.
\end{align}
\end{corollary}
Before we see a nice application of the above identity, let us recall a beautiful identity of Alladi \cite[Theorem 2]{alladi} (see also \cite{ASt}), namely, for $n \in \mathbb{N}$,
\begin{align}\label{alladi}
\sum_{\substack{\pi \in \mathcal{D}(n)\\ s(\pi) \textup{ odd}}} (-1)^{\# (\pi)} =
\left\{
	\begin{array}{ll}
		(-1)^{j-1}  & \mbox{if } n = j^2, \\
		0 & \mbox{if } n \neq j^2.
	\end{array}
\right.
\end{align}
Corollary \ref{coro1} gives us Ono, Schneider, and Wagner's analogue of the above result \eqref{alladi} for $s(\pi)$ even \cite{OSW}.
\begin{theorem}\label{pq}
For $n,j \in \mathbb{N}$ and $k \in \mathbb{Z}\backslash\{0\}$, we have,
\begin{align}
\sum_{\substack{\pi \in \mathcal{D}(n)\\ s(\pi) \textup{ even}}} (-1)^{\# (\pi)} =
\left\{
	\begin{array}{ll}
		(-1)^k-(-1)^{j} & \mbox{if } n=j^2=\frac{k(3k\pm 1)}{2}, \\
		(-1)^{j-1}  & \mbox{if } n = j^2 \mbox{ and } n \mbox{ is not square}, \\
		(-1)^k & \mbox{if } n = \frac{k(3k\pm 1)}{2} \mbox{ and } n \mbox { is not pentagonal}, \\
		0 & \mbox{otherwise}.
	\end{array}
\right.
\end{align}
\end{theorem}

\begin{proof}
Let us take \eqref{coro1}: 
\begin{align}
\sum_{n = 1}^\i \left(\sum_{\substack{\pi \in \mathcal{D}(n)}} (-1)^{\# (\pi)-s(\pi)} \right) q^n=\sum_{n=1}^{\infty}\frac{(-1)^{n-1}q^{n(n+1)/2}}{(q;q)_{n-1}(1+q^n)}=(q;q)_{\infty}-\frac{(q;q)_\infty}{(-q;q)_{\infty}}.\nonumber
\end{align}
Upon using Euler's pentagonal number theorem
\begin{align}
(q;q)_{\infty}=\sum_{j=-\infty}^{\infty}(-1)^jq^{\frac{j(3j+1)}{2}}
\end{align}
and Gauss' result \cite[p.~23]{and}
\begin{align}
\frac{(q;q)_\infty}{(-q;q)_{\infty}}=\sum_{j=-\infty}^{\infty}(-1)^jq^{j^2}
\end{align}
we have
\begin{align*}
\sum_{n = 1}^\i \left(\sum_{\substack{\pi \in \mathcal{D}(n)\\ s(\pi)~even}} (-1)^{\# (\pi)}-\sum_{\substack{\pi \in \mathcal{D}(n)\\ s(\pi)~odd}} (-1)^{\# (\pi)} \right) q^n=\sum_{j=-\infty}^{\infty}(-1)^jq^{\frac{j(3j+1)}{2}}-\sum_{j=-\infty}^{\infty}(-1)^jq^{j^2}.
\end{align*}
Comparing the coefficient of $q^n$ followed by employing \eqref{alladi}, we conclude a proof of the theorem.
\end{proof}

Note that there are infinitely many numbers that are both square \emph{and} pentagonal \cite[Sequence A036353]{O}, though they are very rare.

We now turn our attention to the $k=2$ case in Theorems \ref{Theoremgen1} and \ref{Theoremgen2}. First, if we set $z=1$, we obtain the following analogue of Euler's pentagonal number theorem where now we replace the collection $\mathcal{D}{(n)}$ with the set of distinct partitions of $n$ into atleast two parts, which we refer to as $\mathcal{D}_2 (n)$. 
\begin{corollary}\label{z1} 
If $|q|<1$, then
\begin{align}
\sum_{n = 1}^\i \left(\sum_{\pi \in \mathcal{D}_2(n)} (-1)^{\# (\pi)}\right) q^n=\sum_{n=2}^{\infty}\frac{(-1)^{n}q^{n(n+1)/2}}{(q;q)_{n}}=(q;q)_\infty-\frac{2q-1}{q-1}.
\end{align}
\end{corollary}
(Note the similarity of this corollary to \eqref{euler}.) In addition, we can determine the coefficients of $(q; q)_\infty$ from the pentagonal number theorem. We also have $1/(1 - q) = 1 + q + q^2 + \cdots$ and $2q/(1 - q) = 2q + 2q^2 + 2q^3 + \cdots$. Therefore,
\begin{align*}
\sum_{\pi \in \mathcal{D}_2 (n)} (-1)^{\# (\pi)} q^n & = 1 + \left\{\begin{array}{ll}
(-1)^j & \textrm{ if } n = \frac{j(3j \pm 1)}{2}, \\
0 & \textrm{otherwise.}
\end{array}\right.
\end{align*}
However, if we take $z=-1$, when $k=2$, Theorems \ref{Theoremgen1} and \ref{Theoremgen2} reduce to
\begin{corollary}\label{coro2222}
If $|q|<1$, then
\begin{align}
\sum_{n = 1}^\i \left(\sum_{\pi \in \mathcal{D}_2(n)} (-1)^{\# (\pi)-s_2(\pi)}\right) q^n=
\sum_{n=2}^{\infty}\frac{(-1)^{n}q^{n(n+1)/2}}{(q;q)_{n-2}(1+q^{n-1})(1+q^{n})}=(q;q)_\infty-(q;q^2)_\infty(q^2;q)_\infty.
\end{align}
\end{corollary}
The above identity has a nice combinatorial interpretation.
\begin{theorem}\label{qp}
For $n \geq 1,$
\begin{align}
\sum_{\pi \in \mathcal{D}_2(n)} (-1)^{\# (\pi)-s_2(\pi)}=\left\{
	\begin{array}{ll}
		(-1)^{\ell-1}+(-1)^{j}  & \mbox{if }\ell^2\leq n<(\ell+1)^2~\mbox{and}~n= \frac{j(3j\pm 1)}{2}, \\
		(-1)^{\ell - 1}  & \mbox{if }\ell^2\leq n<(\ell+1)^2~\mbox{and}~n\neq \frac{j(3j\pm 1)}{2}.
	\end{array}
\right.
\end{align}
\end{theorem}
\begin{proof}
Recall the Jacobi triple product identity: 
\begin{align}\label{jtp}
\sum_{n=-\infty}^{\infty}a^{n(n+1)/2}b^{n(n-1)/2}=(a;ab)_\infty(b;ab)_\infty(ab;ab)_\infty.
\end{align} 
Employing the $a = b = q$ case of \eqref{jtp} in Corollary \ref{coro2222} gives us
\begin{align}\label{qn2}
\sum_{n = 1}^\i \left(\sum_{\pi \in \mathcal{D}_2(n)} (-1)^{\# (\pi)-s_2(\pi)}\right) q^n=\sum_{j=-\infty}^{\infty}(-1)^jq^{\frac{j(3j+1)}{2}}-\frac{1}{1-q}\left( 1+2\sum_{n=1}^{\infty}(-1)^nq^{n^2}\right).
\end{align}
Hence, to complete the proof we compare the coefficient of $q^n$ on both of the sides of above equation. Clearly, the coefficient of $q^n$ in
\begin{align*}
\sum_{j = -\i}^\i (-1)^j q^{j(3j + 1)/2}
\end{align*}
is $(-1)^j$ if $n = j(3j \pm 1)/2$ and $0$ otherwise.

The other sum on the righthand side of \eqref{qn2} is trickier. Observe that
\begin{align*}
\frac{1}{1 - q} \sum_{n = 1}^\i (-1)^n q^{n^2} & = (1 + q + q^2 + \cdots)(-q + q^4 - q^9 + \cdots) \\
& = -((q + q^2 + \cdots) - (q^4 + q^5 + \cdots) + (q^9 + q^{10} + \cdots) - \cdots) \\
& = -\sum_{n = 0}^\i ((q^{(2n + 1)^2} + q^{(2n + 1)^2 + 1} + \cdots) - (q^{(2n + 2)^2} + q^{(2n + 2)^2 + 1} + \cdots)) \\
& = -\sum_{n = 0}^\i (q^{(2n + 1)^2} + q^{(2n + 1)^2 + 1} + \cdots + q^{(2n + 2)^2 - 1}).
\end{align*}
The coefficient of $q^n$ in this sum is $-1$ if the largest square $\leq n$ is odd and $0$ otherwise. So, the coefficient of $q^n$ in
\begin{align*}
-\frac{1}{1 - q} \left(1 + 2\sum_{n = 1}^\i (-1)^n q^{n^2}\right)
\end{align*}
is $1$ if the largest square $\leq n$ is odd and $-1$ if it is even.
\end{proof}

To state and prove another important result we need to define
\begin{align}\label{ope}
\partial_zf(z):=z\frac{d}{dz}f(z).
\end{align}
Another application of our Theorem \ref{Theoremgen1} and \ref{Theoremgen2} is in the direction of improving the method of obtaining the closed for evaluation of the sum considered by Uchimura \cite{Uch2} and its further extension by Agarwal et al. \cite{abem} In the later paper, Agarwal et al. considered the a general sum of the type
\begin{align}\label{maji}
M_{m,c}:=\sum_{n=1}^{\infty}n^mc^nq^n(q^{n+1};q)_\infty.
\end{align} 
As mentioned above, the special case $c=1$ of \eqref{maji} was first considered by Uchimura \cite{Uch2}. Agarwal et al. showed that for a non-negative integer $m$ and a complex number $c$ with $|cq|<1,$  then
\begin{align}\label{3.30}
M_{m,c}=\frac{(q;q)_\infty}{(cq)_\infty}Y_{m}\left(K_{1,c},\ldots, K_{m,c} \right),
\end{align}
where $Y_m$ is the Bell polynomial defined as
\begin{align}
Y_{m}\left(u_1,u_2,\ldots,u_m \right):=\sum_{\prod(m)}\frac{m!}{k_1!\cdots k_m!}\left( \frac{u_1}{1!}\right)^{k_1}\cdots\left( \frac{u_m}{m!}\right)^{k_m},
\end{align}
such that $\prod(m)$ denotes a partition of $m$ with $k_1+2k_2+\cdots+mk_m=m$ and 
\begin{align}\label{3.32}
K_{m+1,c}:=\sum_{n=1}^{\infty}\left(\sum_{d|n}d^zc^d \right)q^n=\sum_{n=1}^{\infty}\sigma_{m,c}q^n.
\end{align}

Applying $\partial_z$ to $\textup{FFW}_1$ gives us
\begin{align}\label{spi}
\partial_z \left( \textup{FFW}_1(z,n)\right)=\partial_z\left(\sum_{\pi \in \mathcal{D}(n)} (-1)^{\# (\pi)} z^{s(\pi)}\right)=\sum_{\pi \in \mathcal{D}(n)} (-1)^{\# (\pi)}s(\pi) z^{s(\pi)},
\end{align}
If we apply $\partial_z$ exactly $m$ times, we obtain
\begin{align}\label{spi2}
\partial_z^{(m)} \left( \textup{FFW}_1(z,n)\right) =\partial_z^{(m)} \left(\sum_{\pi \in \mathcal{D} (n)} (-1)^{\# (\pi)} z^{s(\pi)} \right) = \sum_{\pi \in \mathcal{D}(n)} (-1)^{\# (\pi)}s^m(\pi) z^{s(\pi)}.
\end{align}
Note that
\begin{align*}
\sum_{n = 1}^\infty \partial_z^{(m)} \left( \textup{FFW}_1(z,n)\right)  q^n & = \sum_{n = 1}^\infty \left(\sum_{\pi \in \mathcal{D} (n)} (-1)^{\# (\pi)} s^m (\pi) z^{s(\pi)}\right) q^n \\
& = \sum_{n=1}^{\infty}n^kz^nq^n(q^{n+1};q)_\infty.
\end{align*}

Hence, we can use the above observations and Theorems \ref{Theoremgen1} and \ref{Theoremgen2} to derive two new representations of the sum \eqref{maji} for different values of $m$. However, as the value of $m$ increases, the result becomes more complicated. From our next theorem, it is easy to obtain the closed form rather than working with Bell polynomials and the setup defined in \eqref{3.30}-\eqref{3.32}. Let us invoke the operator $\partial_z$ on Theorems \ref{Theoremgen1} and \ref{Theoremgen2}.
\begin{theorem}\label{genth}
For $z\in \mathbb{C}$, $k \in \mathbb{N},$ and $|q|<1$, we have,
\begin{align}
\sum_{n = 1}^\i &\left(\sum_{\pi \in \mathcal{D}(n)} (-1)^{\# (\pi)}s_k(\pi)z^{s_k(\pi)}\right) q^n=z^{k-1}\sum_{j=0}^{k-1}\sum_{n=k}^{\infty}\frac{(-1)^{n-1}q^{n(n+1)/2}}{(q;q)_{n-k}(zq^{n-k+1};q)_k\left(1-zq^{n-j}\right)}\nonumber\\
&=(-1)^kq^{k(k-1)/2}(q;q)_\infty\sum_{j=0}^{k-1}\frac{(-1)^jq^{j(j+1)/2-jk}}{(q;q)_j(q;q)_{k-j-1}}\left(\frac{1}{(zq^{j+1};q)_\infty}\sum_{m=1}^{\infty}\frac{q^{m+j}}{1-zq^{j+m}}-\frac{1}{z}\sum_{n=0}^{k-1}\frac{n(zq^{j+1})^n}{(q;q)_n} \right).
\end{align}
\end{theorem}
If we set $k=1$, we obtain two different representations of \eqref{maji} for $m=1$. 
\begin{corollary}\label{coro3.9}
For $z\in \mathbb{C}$ and $|q|<1$, we have 
\begin{align}
\sum_{n = 1}^\i \left(\sum_{\pi \in \mathcal{D}(n)} (-1)^{\# (\pi)}s(\pi)z^{s(\pi)}\right) q^n&=\sum_{n=1}^{\infty}nz^nq^n(q^{n+1};q)_\infty\nonumber\\
&=\sum_{n=1}^{\infty}\frac{(-1)^{n-1}q^{n(n+1)/2}}{(q;q)_{n-1}\left(1-zq^{n}\right)^2}\nonumber\\
&=\frac{(q;q)_\infty}{(zq;q)_\infty}\sum_{m=1}^{\infty}\frac{q^{m}}{1-zq^{m}}.
\end{align}
\end{corollary}
Note that the sum considered by Agarwal et al. in \eqref{maji} is the sum associated to $s(\pi)$, as we can see in \eqref{spi}. However, one can derive similar results for the $k$th smallest part $s_k(\pi)$ by using Theorem \ref{genth}.
We conclude this section by giving two special cases of Corollary \ref{coro3.9}. First, if we take $z=1$ we obtain \eqref{uchi1}. We obtain a new result by letting $z = -1$ and $k = 1$.
\begin{corollary}
For $|q|<1$, we have 
\begin{align}
\sum_{n = 1}^\i \left(\sum_{\pi \in \mathcal{D}(n)} (-1)^{\# (\pi)+s(\pi)}s(\pi)\right) q^n=\sum_{n=1}^{\infty}\frac{(-1)^{n-1}q^{n(n+1)/2}}{(q;q)_{n-1}(1+q^n)^2}=\frac{(q;q)_\infty}{(-q;q)_\infty}\sum_{n=1}^{\infty}\frac{q^{n}}{1+q^{n}}.
\end{align}
\end{corollary}

\section{Sum-of-tails Identities}\label{Sum-of-tails Identities}

Andrews, Garvan, and Liang \cite{agl} extended the partition statistic $\textup{FFW}(n)$ by a parameter $z$. In particular (see \eqref{agldef}), they defined a function
\begin{align} \label{andrews et al}
\textup{FFW}(z, n) = \sum_{\pi \in \mathcal{D}(n)} (-1)^{\# (\pi) - 1} (1 + z + z^2 + \cdots + z^{s(\pi) - 1}).
\end{align}
In the previous section, we generalized this quantity to $k$th smallest parts by defining $\textup{FFW}_k (z, n)$ \eqref{glvkz} as
\begin{align*}
\textup{FFW}_k (z, n) = \sum_{\pi \in \mathcal{D}(n)} (-1)^{\# (\pi)} z^{s_k (\pi)}.
\end{align*}
We then found the generating function for this function (Theorem \ref{Theoremgen1}).

The natural generalization of \eqref{andrews et al} for $k$th smallest parts is not a sum ending with $z^{s_k (\pi) - 1}$, but rather the sum
\begin{align*}
1 + z + z^2 + \cdots + z^{s_k (\pi) - s_{k - 1} (\pi) - 1}.
\end{align*}
Note that if $k = 1$, then $s_0 (\pi) = 0$, giving us $s_1 (\pi) - s_0 (\pi) = s_1 (\pi)$. So, we can find a sum-of-tails version of Andrews et al.'s weighted partition statistic.

\begin{theorem} \label{z tails} For all $k \geq 2$ and $z \in \mathbb{C}$, we have
\begin{align*}
\sum_{n = 1}^\i \left(\sum_{\substack{\pi \in \mathcal{D} (n) \\ \#(\pi) \geq k}} (-1)^{\# (\pi)} (1 + z + z^2 + \cdots + z^{s_k (\pi) - s_{k - 1} (\pi) - 1})\right) q^n \\
= (-1)^{k - 1} \frac{q^{(k - 1)(k - 2)/2}}{z} \sum_{m = k - 1}^\i \left(\frac{q}{z}\right)^m \left[\begin{array}{c}
m - 1 \\
k - 2
\end{array}\right]_q \sum_{n = m + 1}^\i z^n ((q^n; q)_\i - 1).
\end{align*}
\end{theorem}

\begin{proof} Consider the quantity
\begin{align*}
z^{n - n_{k - 1} - 1} q^{n_1} q^{n_2} \cdots q^{n_{k - 1}} ((1 - q^n)(1 - q^{n + 1}) \cdots - 1) = z^{n - n_{k - 1} - 1} q^{n_1 + n_2 + \cdots + n_{k - 1}} (q^n; q)_\i
\end{align*}
with $n_1 < n_2 < \cdots < n_{k - 1} < n$. This expression counts distinct partitions with at least $k$ parts in which the $k - 1$ smallest terms are $n_1, n_2, \ldots, n_{k - 1}$ and the $k$th smallest part is at least $n$. The coefficient of a given partition $\pi$ is $(-1)^{\# (\pi) - (k - 1)} z^{n - n_{k - 1} - 1}$.

We now sum over all possible $(k - 1)$-tuples $(n_1, n_2, \ldots, n_{k - 1})$ and all $n > n_{k - 1}$. Our sum is
\begin{align} \label{z tails2}
\sum_{n_1 < n_2 < \cdots < n_{k - 1}} q^{n_1 + n_2 + \cdots + n_{k - 1}} \sum_{n = n_{k - 1} + 1}^\i z^{n - n_{k - 1} - 1} ((q^n; q)_\i - 1).
\end{align}
Consider a distinct partition $\pi$ where the $k$th smallest parts are $n_1, n_2, \ldots, n_k$. These partitions only arise when the value of $n$ is the rightmost sum of \eqref{z tails2} lies in the interval $[n_{k - 1} + 1, n_k]$. For a given value of $n$, the corresponding power of $z$ is $z^{n - n_{k - 1} - 1}$. In particular, the exponent for $z$ ranges from $0$ to $n_k - n_{k - 1} - 1 = s_k (\pi) - s_{k - 1} (\pi) - 1$. So, the expression in \eqref{z tails2} is equal to
\begin{align*}
(-1)^{k - 1} \sum_{n = 1}^\i \left(\sum_{\substack{\pi \in \mathcal{D} (n) \\ \# (\pi) \geq k}} (-1)^{\# (\pi)} (1 + z + z^2 + \cdots + z^{s_k (\pi) - s_{k - 1} (\pi) - 1})\right) q^n.
\end{align*}

Our desired quantity is simply $(-1)^{k - 1}$ times the expression in \eqref{z tails2}. Let $m = n_{k - 1}$. We now rewrite the left-hand sum in \eqref{z tails2} as follows:
\begin{align*}
\sum_{n_1 < n_2 < \cdots < n_{k - 2} < m} q^{n_1 + n_2 + \cdots + n_{k - 2} + m} & = \sum_{m = k - 1}^\i q^m \sum_{n_1 < n_2 < \cdots < n_{k - 2} < m} q^{n_1 + n_2 + \cdots + n_{k - 2}} \\
& = \sum_{m = k - 1}^\i q^{m + (k - 1)(k - 2)/2} \left[\begin{array}{c}
m - 1 \\
k - 2
\end{array}\right]_q.
\end{align*}
We now have
\begin{align*}
\sum_{n = 1}^\i \left(\sum_{\substack{\pi \in \mathcal{D} (n)  \\ \# (\pi) \geq k}} (-1)^{\# (\pi)} (1 + z + z^2 + \cdots + z^{s_k (\pi) - s_{k - 1} (\pi) - 1})\right) q^n \\
= (-1)^{k - 1} q^{(k - 1)(k - 2)/2} \sum_{m = k - 1}^\i q^m \left[\begin{array}{c}
m - 1 \\
k - 2
\end{array}\right]_q \sum_{n = m + 1}^\i z^{n - m - 1} ((q^n; q)_\i - 1) \\
= (-1)^{k - 1} \frac{q^{(k - 1)(k - 2)/2}}{z} \sum_{m = k - 1}^\i \left(\frac{q}{z}\right)^m \left[\begin{array}{c}
m - 1 \\
k - 2
\end{array}\right]_q \sum_{n = m + 1}^\i z^n ((q^n; q) - 1).
\end{align*}
\end{proof}

Uchimura also proved the following sum-of-tails identity for $(-1)^{\# (\pi)} s(\pi)$:
\begin{align*}
\sum_{n = 1}^\i \left(\sum_{\pi \in \mathcal{D}(n)} (-1)^{\# (\pi)} s(\pi)\right) q^n = \sum_{n = 1}^\i ((q^n; q)_\i - 1).
\end{align*}
We can generalize this result to the difference between the sums for the $k$th and $(k - 1)$-st smallest parts by plugging $z = 1$ into the previous theorem. We show that this substitution gives us the following result.

\begin{theorem} \label{minus} For all $k \geq 1$, we have
\begin{align*}
\sum_{n = 0}^\i \left(\sum_{\substack{\pi \in \mathcal{D}(n) \\ \#(\pi) \geq k}} (-1)^{\# (\pi)} (s_k (\pi) - s_{k - 1} (\pi))\right) q^n = (-1)^{k - 1} q^{k(k - 1)/2} \sum_{n = k}^\i \left[\begin{array}{c}
n - 1 \\
k - 1
\end{array}\right]_q ((q^n; q)_\i - 1).
\end{align*}
\end{theorem}

\begin{proof} First, observe that if $z = 1$, then
\begin{align*}
1 + z + z^2 + \cdots + z^{s_k (\pi) - s_{k - 1} (\pi) - 1} = s_k (\pi) - s_{k - 1} (\pi).
\end{align*}
Plugging $z = 1$ into Theorem \ref{z tails} shows that our desired sum is
\begin{align*}
(-1)^{k - 1} q^{(k - 1)(k - 2)/2} \sum_{m = k - 1}^\i q^m \left[\begin{array}{c}
m - 1 \\
k - 2
\end{array}\right]_q \sum_{n = m + 1}^\i ((q^n; q)_\i - 1).
\end{align*}
Interchanging the two sums gives us
\begin{align*}
(-1)^{k - 1} q^{(k - 1)(k - 2)/2} \sum_{n = k}^\i ((q^n; q)_\i - 1) \sum_{m = k - 1}^{n - 1} q^m \left[\begin{array}{c}
m - 1 \\
k - 2
\end{array}\right]_q.
\end{align*}
The rightmost sum of the previous equation is
\begin{align*}
\sum_{m = k - 1}^{n - 1} q^m \left[\begin{array}{c}
m - 1 \\
k - 2
\end{array}\right]_q = q^{k - 1} \sum_{m = 0}^{n - k} q^m \left[\begin{array}{c}
m + (k - 2) \\
k - 2
\end{array}\right]_q.
\end{align*}
By \cite[Theorem $3.4$]{and}, this quantity is
\begin{align*}
q^{k - 1} \left[\begin{array}{c}
n - 1 \\
k - 1
\end{array}\right]_q.
\end{align*}
Therefore, our original sum is
\begin{align*}
(-1)^{k - 1} q^{(k - 1)(k - 2)/2} \sum_{n = k}^\i ((q^n; q)_\i - 1) q^{k - 1} \left[\begin{array}{c}
n - 1 \\
k - 1
\end{array}\right]_q \\
= (-1)^{k - 1} q^{k(k - 1)/2} \sum_{n = k}^\i \left[\begin{array}{c}
n - 1 \\
k - 1
\end{array}\right]_q ((q^n; q)_\i - 1).
\end{align*}
\end{proof}

Note that the last formula is similar to \cite[Equation (1.29)]{Gup}. However, the first author's formula has a standard binomial, rather than a $q$-binomial, coefficient.

Using the telescoping sum
\begin{align*}
s_j (\pi) = (s_j (\pi) - s_{j - 1} (\pi)) + (s_{j - 1} (\pi) - s_{j- 2} (\pi)) + \cdots + (s_1 (\pi) - s_0 (\pi)),
\end{align*}
we obtain the following recursive formula for the generating function of $\textup{FFW}_k (n)$.

\begin{theorem} For all $k \geq 2$, we have
\begin{align*}
\sum_{n = 1}^\infty \textup{FFW}_k (n) q^n & = \sum_{n = 1}^\i \sum_{\pi \in \mathcal{D}(n)} \textup{FFW}_{k - 1} (n) q^n \\
& + (-1)^{k - 1} q^{k(k - 1)/2} \sum_{n = k}^\i \left[\begin{array}{c}
n - 1 \\
k - 1
\end{array}\right]_q ((q^n)_\i - 1) \\
& - (-1)^{k - 1} q^{(k - 1)(k - 2)/2} \sum_{n = k - 1}^\i \left[\begin{array}{c}
n - 1 \\
k - 2
\end{array}\right]_q nq^n.
\end{align*}
\end{theorem}

\begin{proof} Theorem \ref{minus} established a formula for the sum of $(-1)^{\# (\pi)} (s_k (\pi) - s_{k - 1} (\pi))$. The difficulty lies in the fact that our two sums impose different conditions on the length of $\pi$. In one we have $\#(\pi) \geq k$, while in the other we have $\# (\pi) \geq k - 1$. However, if $\# (\pi) \geq k - 1$, then $\# (\pi) \geq k$ or $\# (\pi) = k - 1$. So,
\begin{align*}
\sum_{\substack{\pi \in \mathcal{D}(n) \\ \#(\pi) \geq k}} (-1)^{\# (\pi)} s_k (\pi) & = \sum_{\substack{\pi \in \mathcal{D}(n) \\ \#(\pi) \geq k}} (-1)^{\# (\pi)} (s_k (\pi) - s_{k - 1} (\pi)) + \sum_{\substack{\pi \in \mathcal{D}(n) \\ \#(\pi) \geq k}} (-1)^{\# (\pi)} s_{k - 1} (\pi).
\end{align*}
We established the first sum in the previous theorem. As for the second, we have
\begin{align*}
\sum_{\substack{\pi \in \mathcal{D}(n) \\ \#(\pi) \geq k}} (-1)^{\# (\pi)} s_{k - 1} (\pi) & = \sum_{\substack{\pi \in \mathcal{D}(n) \\ \#(\pi) \geq k - 1}} (-1)^{\# (\pi)} s_{k - 1} (\pi) - \sum_{\substack{\pi \in \mathcal{D}(n) \\ \#(\pi) = k - 1}} (-1)^{\# (\pi)} s_{k - 1} (\pi).
\end{align*}
The fact that $\#(\pi)$ is fixed in the rightmost sum makes it easier to bound:
\begin{align*}
\sum_{\substack{\pi \in \mathcal{D}(n) \\ \#(\pi) = k - 1}} (-1)^{\#(\pi)} s_{k - 1} (\pi) & = (-1)^{k - 1} \sum_{\substack{\pi \in \mathcal{D}(n) \\ \#(\pi) = k - 1}} L(\pi),
\end{align*}
where $L(\pi)$ is the \emph{largest} part of $\pi$. The generating function of the rightmost sum is
\begin{align*}
\left(\sum_{\substack{\pi \in \mathcal{D} (n) \\ \# (\pi) = k - 1}} L(\pi)\right) q^n & = \sum_{n = k - 1}^\i \left(\sum_{0 < n_1 < \ldots < n_{k - 2} < n} q^{n_1 + n_2 + \cdots + n_{k - 2}}\right) nq^n \\
& = q^{(k - 1)(k - 2)/2} \sum_{n = k - 1}^\i \left[\begin{array}{c}
n - 1 \\
k - 2
\end{array}\right]_q nq^n.
\end{align*}
Putting everything together gives us
\begin{align*}
\sum_{n = k}^\infty \left(\sum_{\pi \in \mathcal{D}(n)} (-1)^{\# (\pi)} s_k (\pi)\right) q^n & = \sum_{n = k - 1}^\infty \left(\sum_{\pi \in \mathcal{D}(n)} (-1)^{\# (\pi)} s_{k - 1} (\pi)\right) q^n \\
& + (-1)^{k - 1} q^{k(k - 1)/2} \sum_{n = k}^\infty \left[\begin{array}{c}
n - 1 \\
k - 1
\end{array}\right]_q ((q^n; q)_\infty - 1) \\
& - (-1)^{k - 1} q^{(k - 1)(k - 2)/2} \sum_{n = k - 1}^\infty \left[\begin{array}{c}
n - 1 \\
k - 2
\end{array}\right]_q nq^n.
\end{align*}
\end{proof}

The previous result gives us a recursive formula for the generating function of the sum of $(-1)^{\# (\pi)} s_k (\pi)$ over all partitions $\pi \in \mathcal{D}(n)$. By summing all of the terms, we can obtain a sum-of-tails formula for this quantity.

\begin{theorem} For all $k \geq 2$, we have
\begin{align*}
\sum_{n = 1}^\i \textup{FFW}_k (n) q^n & = \sum_{\ell = 2}^\infty (-1)^{\ell (\ell - 1)/2} \left(q^{\ell(\ell - 1)/2} \sum_{n = k}^\i \left[\begin{array}{c}
n - 1 \\
\ell - 1
\end{array}\right]_q ((q^n; q)_\i - 1) \right. \\
& \left. - q^{(\ell - 1)(\ell - 2)/2} \sum_{n = k}^\i \left[\begin{array}{c}
n - 1 \\
\ell - 2
\end{array}\right]_q nq^n\right) + \sum_{n = k}^\i d(n) q^n.
\end{align*}
\end{theorem}

\begin{proof} The previous theorem implies that
\begin{align*}
\sum_{n = 1}^\i \left(\sum_{\pi \in \mathcal{D}(n)} (-1)^{\# (\pi)} (s_k (\pi) - s_{k - 1} (\pi))\right) q^n & = (-1)^{k - 1} \left(q^{k(k - 1)/2} \sum_{n = 1}^\i \left[\begin{array}{c}
n - 1 \\
k - 1
\end{array}\right]_q ((q^n; q)_\i - 1) \right. \\
& \left. - q^{(k - 1)(k - 2)/2} \sum_{n = 1}^\i \left[\begin{array}{c}
n - 1 \\
k - 2
\end{array}\right]_q nq^n\right).
\end{align*}
Summing the corresponding generating functions gives us
\begin{align*}
\sum_{n = 1}^\i \left(\sum_{\pi \in \mathcal{D}(n)} (-1)^{\# (\pi)} s_k (\pi)\right) q^n & = \sum_{\ell = 2}^\i (-1)^{\ell - 1} \left(q^{\ell(\ell - 1)/2} \sum_{n = k}^\i \left[\begin{array}{c}
n - 1 \\
\ell - 1
\end{array}\right]_q ((q^n; q)_\i - 1) \right. \\
& \left. - q^{(\ell - 1)(\ell - 2)/2} \sum_{n = k}^\i \left[\begin{array}{c}
n - 1 \\
\ell - 2
\end{array}\right]_q nq^n\right) + \sum_{n = k}^\i d(n) q^n.
\end{align*}
\end{proof}

Though the $z = -1$ case of Theorem \ref{z tails} does not have as many applications as the $z = 1$ case, it still has some interesting properties. When $z = -1$, we have
\begin{align*}
1 + z + z^2 + \cdots + z^{s_k (\pi) - s_{k - 1} (\pi) - 1} = \left\{\begin{array}{ll}
1, & \textup{if } s_k (\pi) - s_{k - 1} (\pi) \textup{ is odd,} \\
0, & \textup{if } s_k (\pi) - s_{k - 1} (\pi) \textup{ is even.}
\end{array}\right.
\end{align*}
So, we are summing $(-1)^{\# (\pi)}$ on partitions with at least $k$ distinct parts in which the difference between the $k$th and $(k - 1)$-th smallest parts is even. We have
\begin{align} \label{z = -1 sum}
\sum_{n = 1}^\i &\left(\sum_{\substack{\pi \in \mathcal{D}(n) \\ \#(\pi) \geq k \\ s_k (\pi) - s_{k - 1} (\pi) \textup{ is odd}}} (-1)^{\# (\pi)}\right) q^n \nonumber\\
&= (-1)^k q^{(k - 1)(k - 2)/2} \sum_{m = k - 1}^\i q^m \left[\begin{array}{c}
m - 1 \\
k - 2
\end{array}\right]_q \sum_{n = m + 1}^\i (-1)^n ((q^n; q)_\i - 1).
\end{align}

In the case when $k=2$ our restriction is that $s_2 (\pi) - s(\pi)$ is odd and therefore we have,
\begin{align*}
\sum_{n = 1}^\i \left(\sum_{\substack{\pi \in \mathcal{D} (n) \\ \# (\pi) \geq 2 \\ s_2 (\pi) - s_1 (\pi) \textup{ is odd}}} (-1)^{\# (\pi)}\right) q^n & = \sum_{m = 1}^\i q^m \sum_{n = m + 1}^\i (-1)^n ((q^n; q)_\i - 1) \\
& = \sum_{n = 2}^\i (-1)^n ((q^n; q)_\i - 1) \sum_{m = 1}^{n - 1} q^m \\
& = \frac{q}{1 - q}\sum_{n = 2}^\i (-1)^n (1 - q^{n - 1})((q^n; q)_\i - 1).
\end{align*}

\section{Concluding Remarks}

In this article we have initiated the study of the $k$th smallest part $s_k(\pi)$ of a partition into distinct parts. As a special case, we have proved that $s_k(\pi)$ is also connected to divisor function.  We can ask similar questions for the $k$th largest part of a partition into distinct parts or of a general partition. It would also be interesting to find combinatorial proofs of Theorem \ref{pq} and Corollary \ref{qp}.

\end{document}